\documentclass[12pt]{amsart}
\usepackage{graphicx}
\usepackage[T1]{fontenc}
\usepackage{amscd}
\usepackage[top=1in,bottom=1in,left=1.15in,right=1.15in]{geometry} 
\usepackage{amsmath,amssymb,amsfonts}
\usepackage{latexsym}
\usepackage{amsthm}
\usepackage{commath}
\usepackage{color}
\usepackage{hyperref}
\usepackage{tikz}

\newtheorem{theo}{Theorem}[section]
\newtheorem{prop}[theo]{Proposition}
\newtheorem{lemma}[theo]{Lemma}

\newcommand{\Vspace}{\vphantom{\vdots}}

\begin{document}
\title{Polynomial splittings of Ozsv{\'a}th and Szab{\'o}'s $\mathbf{d}$-invariant}
\author{Yuanyuan Bao}
\address{
Institute for Biology and Mathematics of Dynamical Cell Processes (iBMath), Interdisciplinary Center for Mathematical Sciences, Graduate School of Mathematical Sciences, University of Tokyo, 3-8-1 Komaba, Tokyo 153-8914, Japan
}

\email{bao@ms.u-tokyo.ac.jp}
\date{}
\begin{abstract}
For any rational homology 3-sphere and one of its spin$^{c}$-structures,  Ozsv{\'a}th and Szab{\'o} defined a topological invariant, called $d$-invariant. Given a knot in the 3-sphere, the $d$-invariants associated with the prime-power-fold branched covers of the knot, obstruct the smooth sliceness of the knot. These invariants bear some structural resemblances to Casson-Gordon invariants, which obstruct the topological sliceness of a knot. Se-Goo Kim found a polynomial splitting property for Casson-Gordon invariants. In this paper, we show a similar result for Ozsv{\'a}th and Szab{\'o}'s $d$-invariants. We give an application of the result.

\end{abstract}
\maketitle

\section{Introduction}
We work in smooth category, and all manifolds are supposed to be smooth unless stated otherwise. An oriented knot $K$ in the 3-sphere $S^3$ is said to be \textit{slice} if there is a smoothly embedded 2-disk $\Delta$ in the 4-ball $B^4$ satisfying $\partial(B^4,\Delta) = (S^3,K)$. Here $\Delta$ is called a \textit{slice disk} of $K$. A pair of knots $K_1$ and $K_2$ are said to be \textit{smoothly concordant} (which is denoted by $K_{1}\sim K_{2}$) if $K_1\sharp (-K_2)$ is slice where $-K_{2}$ is the mirror image of $K_{2}$ with reversed orientation. Smooth concordance is an equivalence relation among knots and the set of equivalence classes becomes an abelian group under the operation of connected sum. The group is called the knot concordance group and denoted by $C$. Slice knots represent the zero element in $C$.

A knot $K$ is said to be \textit{ribbon} if $K$ bounds a smoothly immersed 2-disc in $S^3$ which has the property that the pre-image of each component of self-intersection consists of a properly embedded arc in the disc and an arc embedded in the interior of the disc.  It is obvious that each ribbon knot is smoothly slice.

To study the group structure of $C$, there are two basic questions to consider. First, given a knot, we need to figure out what the order of the knot is in $C$. Second, given several knots $K_{1}, K_{2}, \cdots K_{n}$, we want to know if they are linearly independent or not in $C$. Finite order elements in $C$ are called torsion elements. As for the first question, the only known torsion in $C$ is 2-torsion, which comes from negative amphicheiral knots. Some invariants such as signature, Rasmussen invariant and Ozsv{\'a}th and Szab{\'o}'s $\tau$-invariant, induce homomorphisms from $C$ to the group of integers. If a knot is a torsion element, it has vanishing value on such invariants.

%for instance, signature, Rasmussen invariant, Ozsv{\'a}th-Szab{\'o}'s $\tau$ invariant, Casson-Gordon invariants, von Neumann $\rho$ invariants, Ozsv{\'a}th and Szab{\'o}'s correction terms and so on and so forth. 

For the independence problem in $C$, there is a systematic way to study it by considering the relative primeness of the Alexander polynomials. Levine \cite{MR0253348} first showed that if the connected sum of two knots with relatively prime Alexander polynomials has vanishing Levine obstructions, then so do both knots. Kim in \cite{MR2127228} showed that the Casson-Gordon-Gilmer obstruction splits in the same manner. In \cite{MR2425750}, a similar splitting property was proved for von Neumann $\rho$ invariants associated with certain metabelian representations.

In this paper, we study a similar splitting property for Ozsv{\'a}th and Szab{\'o}'s $d$-invariants, and apply this property to study the independence problem in $C$. 

Given a 3-manifold $Y$ and one of its torsion spin$^{c}$-structures $s$, the $d$-invariant $d(Y,s)$ is defined for $(Y,s)$ by Ozsv{\'a}th and Szab{\'o} in \cite{MR1957829}. Let $K$ be a knot in $S^3$, and $\Sigma^{n}(K)$ be the $n$-fold cyclic branched cover of $S^3$ along $K$ with $n=q^r$ for some prime number $q$. Then $\Sigma^{n}(K)$ is a rational homology 3-sphere. Therefore we can consider the $d$-invariant of $\Sigma^{n}(K)$ for any of its spin$^{c}$-structures.

If $K$ is a slice knot, let $W^{n}(\Delta)$ be the $n$-fold branched cover of $B^4$ along a slice disk $\Delta$ of $K$. It is known that $W^{n}(\Delta)$ is a rational homology 4-ball whose boundary is $\Sigma^{n}(K)$. As studied in \cite{MR1957829} and reformulated in many papers such as \cite{MR2443966, MR2326940}, many of the $d$-invariants for $\Sigma^{n}(K)$ must vanish (see Theorem~\ref{os}).												

For any spin$^c$-structure $s$ over $\Sigma^{n}(K)$, define $$\bar{d}(\Sigma^{n}(K), s)=d(\Sigma^{n}(K), s)-d(\Sigma^{n}(K), s_{0}),$$ where $s_{0}$ is a special spin$^c$-structure being discussed in Section 2. The first homology group $H_{1}(\Sigma^{n}(K); \mathbb{Z})$ acts freely and transitively on the set of spin$^{c}$-structures $\mathrm{Spin}^c(\Sigma^{n}(K))$. Given an element $s\in \mathrm{Spin}^c(\Sigma^{n}(K))$ and an element $a\in H_{1}(\Sigma^{n}(K); \mathbb{Z})$, let $s+a$ denote the resulting element under the group action. We prove the following theorem.
\begin{theo}
\label{theo1}
Let $K_{1}$ and $K_{2}$ be two knots whose Alexander polynomials are relatively prime in $\mathbb{Q}[t, t^{-1}]$. Suppose further that at least one of $K_{1}$ and $K_{2}$ has non-singular Seifert form. Then the following hold.
\begin{enumerate}
\item If $K_{1}\sharp K_{2}$ is slice, then for all but finitely many prime numbers $q$ there exists a subgroup $M_{i}< H_{1}(\Sigma^{n}(K_{i}); \mathbb{Z})$ satisfying $|M_{i}|^{2}=|H_{1}(\Sigma^{n}(K_{i}))|$ such that $\bar{d}(\Sigma^{n}(K_{i}), s_{0}+m_{i})=0$ for any $m_{i}\in M_{i}$ and $i=1,2$, where $n$ can be any power of $q$.
\item If $K_{1}\sharp K_{2}$ is ribbon, the conclusion holds for any prime number $q$.
\end{enumerate}
\end{theo}

For the $\tau$-invariant defined in \cite{MR2443966}, analogous properties can be proved by using the same argument. Furthermore, for invariants $\mathcal{T}_{p}^{n}(K)$ and $\mathcal{D}_{p}^{n}(K)$ which are defined similarly as $\mathcal{T}_{p}(K)$ and $\mathcal{D}_{p}(K)$ in \cite{MR2443966}, we have Theorem~\ref{2}.		
	
As an application of the results above, we show the following property, which has been know before \cite{MR2127228}.
\begin{prop}
\label{p}
Let $T_{k}$ be the $k$-twist knot. Excluding the unknot, $T_{1}$ (which is the figure-8 knot) and $T_{2}$ (which is Stevedore's knot), no non-trivial linear combinations of twist knots are ribbon knots.

\end{prop}
\subsection*{Acknowledgements}
I would like to thank Charles Livingston and Se-Goo Kim for helpful discussion. My special thanks are due to Livingston for proposing to me the question concerning Lemma~\ref{lemma2}, and for telling me a concise proof that one can choose a slice disk for a ribbon knot which bounds a handlebody in the 4-ball.

The author was supported by Grant-in-Aid for JSPS Fellows, and by Platform for Dynamic Approaches to Living System from the Ministry of Education, Culture, Sports, Science and
Technology, Japan.

\section{Proof of Theorem \ref{theo1}}
\subsection{Alexander polynomial, Seifert form and $d$-invariant}
In this subsection, we review some backgrounds needed in the proof of Theorem \ref{theo1}. Let $K$ be a knot in $S^{3}$ with a Seifert surface $F$ of genus $g$. Define the Seifert form $\theta: H_{1}(F; \mathbb{Z})\times H_{1}(F; \mathbb{Z}) \to \mathbb{Z}$ as $\theta (x, y):=lk(i_{+}(x), y)$ for any simple closed curves $x, y$ representing elements in $H_{1}(F; \mathbb{Z})$, where $i_{+}$ denotes the map that pushes a class off in the positive normal direction. Fix a basis $\lbrace a_{1}, a_{2}, \cdots, a_{2g}\rbrace$ for $H_{1}(F; \mathbb{Z})$ and  let $A$ be the Seifert matrix associated with this basis. 

A Seifert form $\theta$ on $H_{1}(F; \mathbb{Z})$ is said to be \textit{null-concordant} if there exists a direct summand $Z$ of $H_{1}(F; \mathbb{Z})$ so that $\mathrm{rank}(Z)=\dfrac{1}{2}\mathrm{rank}(H_{1}(F; \mathbb{Z}))$ and $\theta (Z, Z)=0$. Such a direct summand $Z$ is called a \textit{metabolizer} of $\theta$.

A knot which has a null-concordant Seifert form is called \textit{algebraically slice}. Two knots $K_{1}$ and $K_{2}$ are said to be algebraically concordant if $K_{1}\sharp (-K_{2})$ is algebraically slice. The set of the equivalence classes of knots under this relation becomes a group as well, called \textit{algebraic concordance group}, and we denote it $C_{alg}$.

The following lemma is Lemma 3.1 in \cite{MR2127228}, which was refined from \cite[Proposition 3]{MR0283786}. Note that the definition of null-concordance of a Seifert form in \cite{MR2127228} is different from the one we are using, but the two definitions are equivalent.

\begin{lemma}[Kervaire, Levine, Kim]
\label{kim}
Given two knots $K_{1}$ and $K_{2}$, let $F_{i}$ be a Seifert surface of $K_{i}$, and $\theta_{i}$ be the Seifert form on $H_{1}(F_{i}; \mathbb{Z})$ for $i=1,2$. Suppose the Alexander polynomials of $K_{1}$ and $K_{2}$ are relatively prime in $\mathbb{Q}[t^{-1}, t]$, and either $\theta_{1}$ or $\theta_{2}$ is non-singular.
Then if $\theta_{1}\oplus\theta_{2}$ is null-concordant with a metabolizer $Z$, then $\theta_{i}$ is null-concordant with metabolizer $Z_{i}=Z\cap H_{1}(F_{i}; \mathbb{Z})$ for $i=1, 2$.
\end{lemma}

From this lemma we see that,  with the assumption in this lemma, if $K_{1}\sharp K_{2}$ is algebraically slice, so are both $K_{1}$ and $K_{2}$.

%Given a rational homology 3-sphere $Y$ and a spin$^{c}$-structure $s\in \mathrm{Spin}^c(Y)$, Ozsv{\'a}th and Szab{\'o} defined an invariant $d(Y,s)$, called correction term or $d$-invariant for $(Y, s)$. Given a knot $K$ and an integer $n=q^{r}$ for some prime number $q$, let $\Sigma^{n}(K)$ denote the $n$-fold branched cover of $S^{3}$ along $K$. It is a rational homology 3-sphere, so we can consider the correction terms associated with it.

Let $\Sigma^{n}(K)$ be defined as before for a given knot $K$ and a prime power $n$.
The homology group $H_{1}(\Sigma^{n}(K); \mathbb{Z})$ acts freely and transitively on the set $\mathrm{Spin}^{c}(\Sigma^{n}(K))$. So a choice of one spin$^{c}$-structure gives a bijection between $\mathrm{Spin}^{c}(\Sigma^{n}(K))$ and $H_{1}(\Sigma^{n}(K); \mathbb{Z})$. As discussed in Section 2 of \cite{MR2443966}, there exists a canonical spin$^{c}$-structure $s_{0}\in \mathrm{Spin}^{c}(\Sigma^{n}(K))$,  which is uniquely characterised by $K$ and $n$. For more details about the definition, please refer to \cite{MR2443966} and \cite{jabuka}. If $H_{1}({\Sigma^{n}(K)}; \mathbb{Z})$ has no 2-torsion, then $s_{0}$ is the unique spin-structure over $\Sigma^{n}(K)$. Under this $s_{0}$, we can identify $H_{1}(\Sigma^{n}(K))$ and $\mathrm{spin}^{c}(\Sigma^{n}(K))$, by sending $m\in H_{1}(\Sigma^{n}(K); \mathbb{Z})$ to $s_{0}+m \in \mathrm{Spin}^{c}(\Sigma^{n}(K))$. One nice property of $s_{0}$ is that it is compatible with the connected sum of knots. Namely, $$s_{0}(K_{1}\sharp K_{2})=s_{0}(K_{1})\sharp s_{0}(K_{2})$$ for two given knots $K_{1}$ and $K_{2}$, where $s_{0}(K)$ denotes the spin$^{c}$-structure $s_{0}$ of $\Sigma^{n}(K)$.

If the knot $K$ is slice, as before let $W^{n}(\Delta)$ be the $n$-fold branched cover of $B^{4}$ along a slice disk $\Delta$ of $K$. Consider the homomorphism $\zeta: H_{1}(\Sigma^{n}(K); \mathbb{Z})\rightarrow H_{1}(W^{n}(\Delta); \mathbb{Z})$ induced by the inclusion map. Then it is known that a spin$^{c}$-structure over $\Sigma^{n}(K)$ can be extended to $W^{n}(\Delta)$ if and only if $s$ has the form $s=s_{0}+m$ for some $m\in \mathrm{Ker}(\zeta)$. As an application of the results in \cite{MR1957829}, the following theorem is known (see also \cite{MR2443966}).

\begin{theo}[Ozsv{\'a}th and Szab{\'o}]
\label{os}
If $K$ is slice, then $d(\Sigma^{n}(K), s_{0}+m)=0$ for any $m\in \mathrm{Ker}(\zeta)$.\end{theo}

\subsection{Proof of Theorem \ref{theo1}}
The order of $H_{1}(\Sigma^{n}(K); \mathbb{Z})$ is determined by the Alexander polynomial $\Delta_{K}(t)$ of $K$. Precisely we have Fox's formula $$\left| H_{1}(\Sigma^{n}(K); \mathbb{Z})\right|=\left|\prod_{j=0}^{n}\Delta_{K}(\mathrm{exp}({2\pi j}/{n}))\right|.$$ We prove the following lemma.

\begin{lemma}
\label{lemma2}
Given a knot $K$ and a prime number $p\geq 2$, let $S_{p, K}$ be the set of prime numbers with the property that if $q\in S_{p, K}$ then there is certain integer $r$ such that $p$ divides the order of $H_{1}(\Sigma^{q^r}(K); \mathbb{Z})$. Then $S_{p,K}$ is a finite set.
\end{lemma}
\begin{proof}
Given a prime power $n$, consider the resultant of $\Delta_{K}(t)$ and $t^{n}-1$ over the field of complex numbers $\mathbb{C}$. Then we have
$$\mathrm{Res}(\Delta_{K}(t), t^{n}-1)=a_{\Delta}^{n}\prod_{j=0}^{n}\Delta_{K}(\mathrm{exp}({2\pi j}/{n})),$$ where $a_{\Delta}$ is the leading coefficient of $\Delta_{K}(t)$. If $p$ divides the order of $H_{1}(\Sigma^{n}(K); \mathbb{Z})$, then $p$ divides $\mathrm{Res}(\Delta_{K}(t), t^{n}-1)$. Therefore $\mathrm{Res}(\Delta_{K}(t), t^{n}-1)=0$ in the field $\mathbb{F}_{p}$, which means $\Delta_{K}(t)$ and $t^{n}-1=(t-1)(t^{n-1}+t^{n-2}+\cdots+1)$ having a common root in the algebraic closure of $\mathbb{F}_{p}$. Since $1$ is never a root of $\Delta_{K}(t)$, we see that $\Delta_{K}(t)$ and $t^{n-1}+t^{n-2}+\cdots+1$ have a common root in the algebraic closure of $\mathbb{F}_{p}$.

Since $p\geq 2$, we have $1\notin S_{p,K}$. 
We assume that the set $S_{p,K}$ is an infinite set. Since $\Delta_{K}(t)$ has only finitely many roots in the algebraic closure of $\mathbb{F}_{p}$, there must be two elements $q_{1}$ and $q_{2}$ in $S_{p,K}$ such that $t^{q_{1}^{r_{1}}-1}+t^{q_{1}^{r_{1}}-2}+\cdots+1$ and $t^{q_{2}^{r_{2}}-1}+t^{q_{2}^{r_{2}}-2}+\cdots+1$ have a common root in the algebraic closure of $\mathbb{F}_{p}$ for some positive integers $r_{1}$ and $r_{2}$. This contradicts Lemma~\ref{lemma3}. Therefore opposed to our assumption, the set $S_{p,K}$ is a finite set.

\end{proof}

\begin{lemma}
\label{lemma3}
If $m$ and $n$ are relatively prime integers greater than or equal to two, $t^{m-1}+t^{m-2}+\cdots+1$ and $t^{n-1}+t^{n-2}+\cdots+1$ can never have a common root in the algebraic closure of $\mathbb{F}_{p}$.
\end{lemma}
\begin{proof}
It is enough to check that $\mathrm{Res}(t^{m-1}+t^{m-2}+\cdots+1, t^{n-1}+t^{n-2}+\cdots+1)=\pm 1$, which is non-zero in the algebraic closure of $\mathbb{F}_{p}$. In fact we have
\begin{equation*}
\label{local}
\tag{\dag}
\begin{split}
&\mathrm{Res}(t^{m-1}+t^{m-2}+\cdots+1, t^{n-1}+t^{n-2}+\cdots+1)\\ 
&=\det (A(m,n))=\pm 1,
\end{split}
\end{equation*} 
where 
\begin{align*}
A(m,n)
=\left({\begin{array}{ccccccc}
1 & 1 & 1 & \cdots & 0 & 0 & 0\\
0 & 1 & 1 & \cdots & 0 & 0 & 0\\
\vdots & \vdots & \vdots & & \vdots & \vdots & \vdots\\
0 & 0 & 0 & \cdots & 1 & 1 & 0\\
0 & 0 & 0 & \cdots & 1 & 1 & 1\\
\hline
1 & 1 & 1 & \cdots & 0 & 0 & 0\\
0 & 1 & 1 & \cdots & 0 & 0 & 0\\
\vdots & \vdots & \vdots & & \vdots & \vdots & \vdots\\
0 & 0 & 0 & \cdots & 1 & 1 & 0\\
0 & 0 & 0 & \cdots & 1 & 1 & 1\\
\end{array}} \right)
  \begin{array}{@{\kern-\nulldelimiterspace}l@{}}
    \left.\begin{array}{@{}c@{}}\Vspace\\\Vspace\\\Vspace\\\Vspace\end{array}\right\}\\
    \left.\begin{array}{@{}c@{}}\Vspace\\\Vspace\\\Vspace\\\Vspace\end{array}\right\}
  \end{array}
{\begin{array}{ccccccc}
\\
n-1 \text{ rows,}\\
\text{with } m \text{ 1's on each row;}\\
\\
\\
\\
\\
m-1 \text{ rows,}\\
\text{with } n \text{ 1's on each row.}\\
\\
\end{array}}. 
\end{align*}

The first equation in (\ref{local}) is a basic property of the resultant of two polynomials. We prove the second equation by induction on $|m-n|$. If $|m-n|=1$, we assume that $n=m+1$ and subtract the $j$'s row from the $n+j-1$'s row in $A(m,n)$, where $1\leq j \leq m-1$. Then we have 
\begin{align*}
\det(A(m,n))=
\left({\begin{array}{cccc|cccc}
1 & 1 & 1 & \cdots &  &  & \\
 & 1 & 1 & \cdots &  &  & *\\
 &  & \ddots & &  &  & \\
 0&  &  &  1 &  & \\
\hline
 &  &  & &1&  &  & \\
&  &  & && 1 &  & \\
&  &  0& &&  & \ddots & \\
&  &  & &&  &  & 1\\
\end{array}} \right)=1.
\end{align*}

If $|m-n|\geq 2$, assume $n>m$ and subtract the $j$'s row from the $n+j-1$'s row, where $1\leq j \leq m-1$. The resulting matrix is 
\begin{align*}
\left({\begin{array}{cccc|ccc}
1 & 1 & 1 & \cdots &  &  & \\
 & 1 & 1 & \cdots &  &  *& \\
 &  & \ddots & &  &  & \\
 0&  &  &  1 &  & \\
\hline
 &  &  & &  &  & \\
 &  & {\huge 0} & & & A(m,n-m) & \\
\end{array}} \right).
\end{align*}
Therefore $\det(A(m,n))=\det (A(m,n-m))$.\\
(1) If $|m-(n-m)|<n-m$, it follows from our induction that $\det(A(m,n))=\det (A(m,n-m))=\pm 1$.\\
(2) If $|m-(n-m)|\geq n-m$, then $2m-n\geq n-m \geq 2$, which implies $2n/3\leq m<n$. Let $k$ be the integer for which $kn/(k+1)\leq m<(k+1)n/(k+2)$ for $k\geq 2$. Then we have 
\begin{align*}
\det(A(m,n))&=\det (A(m,n-m))=\pm \det (A(n-m,2m-n))\\
&=\pm \det (A(n-m,3m-2n))=\cdots\\
&=\pm \det (A(n-m,km-(k-1)n)).
\end{align*}
The condition $kn/(k+1)\leq m<(k+1)n/(k+2)$ implies that $km-(k-1)n\geq n-m$ and $[km-(k-1)n-(n-m)]<(n-m)$. It follows from our induction that $\det(A(m,n))=\pm \det (A(n-m,km-(k-1)n))=\pm 1$. 
\end{proof}

\begin{proof}[Proof of Theorem \ref{theo1}]
\textbf{(i)} Choose a Seifert surface $F_{i}$ for the knot $K_{i}$ for $i=1,2$. Let $K=K_{1}\sharp K_{2}$ and $F=F_{1}\sharp F_{2}$. Then $F$ is a Seifert surface of $K$. Suppose $K$ is a slice knot. Then $F\cup \Delta$ bounds a 3-manifold in $B^{4}$, denoted by $R$, where $\Delta$ is a slice disk of $K$ in the 4-ball $B^{4}$. Consider the map $\iota: H_{1}(F; \mathbb{Z})\rightarrow H_{1}(R; \mathbb{Z})/\mathrm{Tor}$ induced by inclusion where $\mathrm{Tor}$ is the torsion part of $H_{1}(R; \mathbb{Z})$, and let $Z=\mathrm{ker}(\iota)$. Then $Z$ is a metabolizer of the Seifert form $\theta$ associated with $F$. See Theorems~3.1.1 and 3.1.2 in \cite{charles} for detailed discussion.

Let $Y$ denote $S^{3}$ sliced along $F$, and $X$ denote $D^{4}$ sliced along $R$. Considering the construction of $\Sigma^{n}(K)$ and $W^{n}(\Delta)$, we have the following commutative diagram.
\begin{center}
\begin{tikzpicture}
  \node (A) {};
  \node (B) [node distance=2.5cm, right of=A] {$\bigoplus_{1\leqslant i \leqslant n}H_{1}(F; \mathbb{Z})$};
  \node (C) [node distance=4cm, right of=B] {$\bigoplus_{1\leqslant i \leqslant n}H_{1}(Y; \mathbb{Z})$};
  \node (D) [node distance=3.7cm, right of=C] {$H_{1}(\Sigma^{n}(K); \mathbb{Z})$};
  \node (E) [node distance=2cm, right of=D] {$0$};
  \node (F) [node distance=2.5cm, below of=A] {};
  \node (G) [node distance=2.5cm, below of=B] {$\bigoplus_{1\leqslant i \leqslant n}H_{1}(R; \mathbb{Z})$};
  \node (H) [node distance=2.5cm, below of=C] {$\bigoplus_{1\leqslant i \leqslant n}H_{1}(X; \mathbb{Z})$};
  \node (I) [node distance=2.5cm, below of=D] {$H_{1}(W^{n}(\Delta); \mathbb{Z})$};
  \node (J) [node distance=2.5cm, below of=E] {$0$};

  \draw[->] (A) to node {} (B);
  \draw[->] (B) to node [below]{$j$} (C);
  \draw[->] (C) to node [below]{} (D); 
  \draw[->] (D) to node {} (E); 
  \draw[->] (F) to node {} (G);
  \draw[->] (G) to node {} (H);
  \draw[->] (H) to node [above] {} (I); 
  \draw[->] (I) to node {} (J); 
  \draw[->] (B) to node {} (G); 
  \draw[->] (C) to node {} (H); 
   \draw[->] (D) to node {} (I); 
\end{tikzpicture}
\end{center}
The map $j$ and the maps in the vertical direction are induced by the inclusion maps. We have the following isomorphisms.
$$H_{1}(Y; \mathbb{Z})\cong H_{1}(S^{3}-F; \mathbb{Z})\cong H^{1}(F; \mathbb{Z})\cong \mathrm{Hom}(H_{1}(F; \mathbb{Z}), \mathbb{Z})\cong H_{1}(F; \mathbb{Z}).$$ The second isomorphism follows from Alexander duality. In the same vein, we can establish the isomorphism $H_{1}(X; \mathbb{Z})\cong H_{1}(R; \mathbb{Z})$. So we can replace the previous commutative diagram with the following diagram, which we denote by $(*)$.
\begin{center}
\begin{tikzpicture}
  \node (A) {};
  \node (B) [node distance=2.5cm, right of=A] {$\bigoplus_{1\leqslant i \leqslant n}H_{1}(F; \mathbb{Z})$};
  \node (C) [node distance=4cm, right of=B] {$\bigoplus_{1\leqslant i \leqslant n}H_{1}(F; \mathbb{Z})$};
  \node (D) [node distance=3.7cm, right of=C] {$H_{1}(\Sigma^{n}(K); \mathbb{Z})$};
  \node (E) [node distance=2cm, right of=D] {$0$};
    \node (M) [node distance=2cm, right of=E] {};
  \node (F) [node distance=2.5cm, below of=A] {};
  \node (G) [node distance=2.5cm, below of=B] {$\bigoplus_{1\leqslant i \leqslant n}H_{1}(R; \mathbb{Z})$};
  \node (H) [node distance=2.5cm, below of=C] {$\bigoplus_{1\leqslant i \leqslant n}H_{1}(R; \mathbb{Z})$};
  \node (I) [node distance=2.5cm, below of=D] {$H_{1}(W^{n}(\Delta); \mathbb{Z})$};
  \node (J) [node distance=2.5cm, below of=E] {$0$};
  \node (N) [node distance=2.5cm, below of=M] {};
  
  \draw[->] (A) to node {} (B);
  \draw[->] (B) to node [below] {$f$} (C);
  \draw[->] (C) to node [below] {$g$} (D); 
  \draw[->] (D) to node {} (E); 
  \draw[->] (F) to node {} (G);
  \draw[->] (G) to node {} (H);
  \draw[->] (H) to node [above] {$h$} (I); 
  \draw[->] (I) to node {} (J); 
  \draw[->] (B) to node {} (G); 
  \draw[->] (C) to node [right] {$\bigoplus_{1\leqslant i \leqslant n}\bar{\iota}$} (H); 
  \draw[->] (D) to node [right] {$\zeta$} (I); 
\end{tikzpicture}
\end{center}
the maps $\bar{\iota}$ and $\zeta$ are induced by the inclusion maps.

We fix a basis $\lbrace a_{1}, a_{2}, \cdots, a_{2g}\rbrace$ for $H_{1}(F; \mathbb{Z})$ and let $A$ be the Seifert matrix of $\theta$ associated with this basis. Now we see there is a basis of $\bigoplus_{1\leqslant i \leqslant n}H_{1}(F; \mathbb{Z})$ which is naturally induced by $\lbrace a_{1}, a_{2}, \cdots, a_{2g}\rbrace$. Under this basis, the map $f$ is represented by the matrix
\begin{eqnarray*}
f=\left({\begin{array}{cccccc}
G & I-G & 0 & 0 & \cdots & 0\\
0 & G & I-G & 0 & \cdots & 0\\
0 & 0 & G & I-G & \cdots & 0\\
& \vdots & &\vdots & &\\
I-G & 0 & 0 &0 & \cdots & G \\
\end{array}} \right),
\end{eqnarray*}
where $G=(A^{t}-A)^{-1}A^{t}$ (See discussion before \cite[Lemma 1]{MR1201199} or \cite[Theorem~6.2.2]{charles}). Here we abuse $f$ to denote both the map and the matrix. It is known that $f$ is a presentation matrix of $H_{1}(\Sigma^{q}(K); \mathbb{Z})$.
Define $f_{1}$ and $f_{2}$ for $K_{1}$ and $K_{2}$ respectively. Then $f=f_{1}\oplus f_{2}$.

Remember $Z$ is the kernel of the map $\iota: H_{1}(F; \mathbb{Z})\rightarrow H_{1}(R; \mathbb{Z})/\mathrm{Tor}$. Let $M=g(\bigoplus_{1\leqslant i \leqslant n} Z)$. 
By Lemma~\ref{lemma4}, the order of $M$ is the square root of the order of $H_{1}(\Sigma^{n}(K); \mathbb{Z})$. 
By Lemma~\ref{kim}, the metaboizer $Z$ has a splitting $Z=Z_{1}\oplus Z_{2}$ where $Z_{i}$ is a metabolizer of the Seifert form $\theta_{i}$ associated with $F_{i}$, for $i=1,2$. Then $M$ has a splitting $M=M_{1}\oplus M_{2}$ where $M_{i}=g(\bigoplus_{1\leqslant i \leqslant n} Z_{i})$ for $i=1,2$. By Lemma~\ref{lemma4} the order of $M_{i}$ is the square root of the order of $H_{1}(\Sigma^{n}(K_{i}); \mathbb{Z})$ for $i=1,2$.

Since $H_{1}(R; \mathbb{Z})$ is finitely generated, there are only finitely many prime numbers dividing the order of $\mathrm{Tor}$, say they are elements in $\lbrace p_{1}, p_{2}, \cdots, p_{s}\rbrace$. We let $S=\bigcup_{j=1}^{s}S_{p_{j},K}$ where $S_{p_{j}, K}$ is the set described in Lemma~\ref{lemma2}, and then $S$ is again a finite set. 

Remember that $n=q^{r}$ for some prime number $q$ and positive integer $r$. Suppose the prime number $q$ is not in the set $S$. In this case, we claim that 
$$h(\bigoplus_{1\leqslant i \leqslant n} \bar{\iota} (\bigoplus_{1\leqslant i \leqslant n} Z))=h(\bigoplus_{1\leqslant i \leqslant n} (\bar{\iota} (Z)))=0.$$
Since $Z$ is the kernel of the map $\iota: H_{1}(F; \mathbb{Z})\rightarrow H_{1}(R; \mathbb{Z})/\mathrm{Tor}$, so $\bar{\iota} (Z)$ belongs to the torsion part $\mathrm{Tor}$ of $H_{1}(R; \mathbb{Z})$. Given $x\in \bigoplus_{1\leqslant i \leqslant n} Z$, let $y=\bigoplus_{1\leqslant i \leqslant n} \bar{\iota}(x)$. Then the order of $y$ divides the order of $\mathrm{Tor}$. On the other hand, the order of $g(x)$ divides the order of $H_{1}(\Sigma^{n}(K); \mathbb{Z})$. By the commutativity of the diagram (*) above, $\zeta (g(x))=h(y)$. If $\zeta (g(x))=h(y)\neq 0$, the order of $h(y)$ divides both the orders of $\mathrm{Tor}$ and $H_{1}(\Sigma^{n}(K); \mathbb{Z})$. So there exists an element $p\in \lbrace p_{1}, p_{2}, \cdots, p_{s}\rbrace$ dividing the order of $H_{1}(\Sigma^{n}(K); \mathbb{Z})$. This conflicts with the choice of $q$. Therefore $\zeta (g(x))=h(y)=0$, and our claim is proved.

By the commutativity of the diagram ($*$), we have $M\subset \mathrm{Ker}(\zeta)$. Note that since $K$ is slice, the order of $\mathrm{Ker}(\zeta)$ is the square root of the order of $H_{1}(\Sigma^{n}(K); \mathbb{Z})$ by Lemma~3 in \cite{MR900252}. Therefore $M$ and $\mathrm{Ker}(\zeta)$ has the same order as finite groups, so $M=\mathrm{Ker}(\zeta)$. The $d$-invariants defined on $M=\mathrm{Ker}(\zeta)$ are zero by Theorem~\ref{os}.

We now show that the $\bar{d}$-invariants of $\Sigma^{n}(K_{i})$ defined on $M_{i}$ are zero for both $i=1,2$. For any element $m_{1}\in M_{1}$, the element $(m_{1}, 0)$ is included in $M$, so by the additivity of $d$-invariant we have
\begin{eqnarray*}
d(\Sigma^{n}(K), s_{0}+(m_{1}, 0))
=d(\Sigma^{n}(K_{1}), s_{0}+m_{1})+d(\Sigma^{n}(K_{2}), s_{0})=0.
\end{eqnarray*}
Here we abuse $s_{0}$ to denote the unique spin$^{c}$-structures discussed in Section~2.1 over $\Sigma^{n}(K_{1})$, $\Sigma^{n}(K_{2})$ or $\Sigma^{n}(K)$. The value $d(\Sigma^{n}(K_{1}), s_{0}+m_{1})=-d(\Sigma^{n}(K_{2}), s_{0})$ does not depend on the choice of $m_{1}$, so we have
$$\bar{d}(\Sigma^{n}(K_{1}), s_{0}+m_{1})=d(\Sigma^{n}(K_{1}), s_{0}+m_{1})-d(\Sigma^{n}(K_{1}), s_{0})=0$$ for any $m_{1}\in M_{1}$. The same fact can be proved for $K_{2}$.

\textbf{(ii)} If $K$ is a ribbon knot, we can choose $R$ to be a handlebody, in which case $H_{1}(R; \mathbb{Z})$ is torsion free. Then the set $S$ is an empty set. Therefore the conclusion in $(i)$ holds for any prime power $n$.

\end{proof}

In the rest of this section, we give a proof of the following lemma, which we cannot find a good reference for it. 

\begin{lemma}
\label{lemma4}
Suppose $Z$ is a metabolizer of the Seifert form $\theta$ for a Seifert surface $F$ of the knot $K$.
The order of $M=g(\bigoplus_{1\leqslant i \leqslant n} Z)$ is the square root of the order of $H_{1}(\Sigma^{n}(K); \mathbb{Z})$ for any prime power $n$. 
\end{lemma}

Seifert \cite{seifert} proved that $G^{n}-(G-I)^{n}$ is a presentation matrix for $H_{1}(\Sigma^{n}(K); \mathbb{Z})$ with the set of generators $\lbrace a_{1}, a_{2}, \cdots, a_{2g}\rbrace$. Namely we have an exact sequence 
$$0\to H_{1}(F; \mathbb{Z})\xrightarrow{\hat{f}:= G^{n}-(G-I)^{n}} H_{1}(F; \mathbb{Z})\xrightarrow{\hat{g}} H_{1}(\Sigma^{n}(K); \mathbb{Z})\to 0.$$
The map $\hat{g}$ induces an isomorphism $H_{1}(\Sigma^{n}(K); \mathbb{Z})\cong H_{1}(F; \mathbb{Z})/\mathrm{Im}(\hat{f})$ and $\hat{g}(Z)$ is isomorphic to $Z/(\mathrm{Im}(\hat{f})\cap Z)$. 

We prove that the order of $Z/(\mathrm{Im}(\hat{f})\cap Z)$ is a square root of that of $H_{1}(F; \mathbb{Z})/\mathrm{Im}(\hat{f})$.
The proof is similar to that of \cite[Lemma 2]{MR1201199}. 
As stated there, we may extend a basis $\lbrace x_{1}, x_{2}, \cdots, x_{g}\rbrace$ of $Z$ to a basis $\lbrace x_{1}, x_{2}, \cdots, x_{g}, y_{1}, y_{2}, \cdots, y_{g}\rbrace$ of $H_{1}(F; \mathbb{Z})$. Under this basis, the matrices $A$, $G$ and $\hat{f}$ have the forms 
$$
A=\left({\begin{array}{cc}
0 & C+I\\
C' & E  \\
\end{array}} \right),
G=\left({\begin{array}{cc}
C'+I & E'\\
0 & -C  \\
\end{array}} \right)
 \text{ and }
\hat{f}=\left({\begin{array}{cc}
C_{n}' & *\\
0 & C_{n}  \\
\end{array}} \right),
$$
where $C_{n}=C^{n}-(C-I)^{n}$. By the invertibility of $C_{n}$, we have $\mathrm{Im}(\hat{f})\cap Z\cong C_{n}'(Z)$. Therefore the order of $Z/(\mathrm{Im}(\hat{f})\cap Z)$ is $|\det(C_{n}')|$, while the order of $H_{1}(F; \mathbb{Z})/\mathrm{Im}(\hat{f})$ is $|\det(\hat{f})|=|\det(C_{n}')|^{2}$. 

%The form of $G$ tells us that $GH\subset H$.

\begin{proof}[Proof of Lemma~\ref{lemma4}]
We show that $Z/(\mathrm{Im}(\hat{f})\cap Z)$ is isomorphic to $(\bigoplus_{1\leqslant i \leqslant n} Z)/(\mathrm{Im}(f)\cap \bigoplus_{1\leqslant i \leqslant n} Z)$, which is isomorphic to $g(\bigoplus_{1\leqslant i \leqslant n} Z)$. Following Lemma~1 in \cite{MR1201199}, there are integral determinant $\pm 1$ $2gn\times 2gn$ matrices $R$ and $C$ which can be written as block matrices whose blocks are polynomials in $G$, such that $f^{+}=RfC$ has the form
\begin{eqnarray*}
f^{+}=\left({\begin{array}{cccccc}
I & * & \cdots &  *  &* \\
0 & I &\cdots  & *  &* \\
 \vdots&  & \ddots & \vdots  &\vdots \\
0& 0& \cdots&I & *\\
 0&  0&  \cdots&0& \hat{f} \\
\end{array}} \right),
\end{eqnarray*}
where the stars mean some unspecified polynomials in $G$.
It is very easy to check that $(\bigoplus_{1\leqslant i \leqslant n} Z)/(\mathrm{Im}(f^{+})\cap \bigoplus_{1\leqslant i \leqslant n} Z)\cong Z/(\mathrm{Im}(\hat{f})\cap Z)$ by the forms of $f^{+}$ and $\hat{f}$.

Next we show that $$(\bigoplus_{1\leqslant i \leqslant n} Z)/(\mathrm{Im}(f^{+})\cap \bigoplus_{1\leqslant i \leqslant n} Z)\cong(\bigoplus_{1\leqslant i \leqslant n} Z)/(\mathrm{Im}(f)\cap \bigoplus_{1\leqslant i \leqslant n} Z)$$ by using the properties of $R$ and $C$. 
Remember that $R$ and $C$ are automorphisms of $\bigoplus_{1\leqslant i \leqslant n}H_{1}(F; \mathbb{Z})$, so $\mathrm{Im}(f^{+})=\mathrm{Im}(RfC)=\mathrm{Im}(Rf)$. Now we only need to show that $R$ induces an isomorphism between $(\bigoplus_{1\leqslant i \leqslant n} Z)/(\mathrm{Im}(f)\cap \bigoplus_{1\leqslant i \leqslant n} Z)$ and $(\bigoplus_{1\leqslant i \leqslant n} Z)/(\mathrm{Im}(Rf)\cap \bigoplus_{1\leqslant i \leqslant n} Z)$. We show that 
\begin{equation*}
\label{relation}
\tag{\ddag}
\begin{split}
&R(\bigoplus_{1\leqslant i \leqslant n} Z)=\bigoplus_{1\leqslant i \leqslant n} Z, \text{ and } \\
&R(\mathrm{Im}(f)\cap \bigoplus_{1\leqslant i \leqslant n} Z)\cong \mathrm{Im}(Rf)\cap \bigoplus_{1\leqslant i \leqslant n} Z.
\end{split}
\end{equation*}

%Remember that $Z$ is invariant under the map $G$ and $R$ and $C$ are automorphisms of $\bigoplus_{1\leqslant i \leqslant n}H_{1}(F; \mathbb{Z})$. So $R(\bigoplus_{1\leqslant i \leqslant n} Z)=C(\bigoplus_{1\leqslant i \leqslant n} Z)=\bigoplus_{1\leqslant i \leqslant n} Z$, and our lemma follows.

Choose an order for the elements in the basis of $\bigoplus_{1\leqslant i \leqslant n}H_{1}(F; \mathbb{Z})$ so that the elements in $\bigoplus_{1\leqslant i \leqslant n} Z$ take the first $ng$ positions. Remember that the blocks of $R$ are polynomials in $G$. The form of $G$ tells us that under the reordered basis, the matrix $R$ and its inverse are $2ng\times 2ng$ matrices with the form
$$
\left({\begin{array}{cc}
* & *\\
0 & *  \\
\end{array}} \right),
$$ where stars are $ng\times ng$ matrices. Since $R$ and its inverse are automorphisms, it is now easy to check that relations ($\ref{relation}$) hold.

\end{proof}

\subsection{$\tau$-invariant, $\mathcal{T}_{p}^{n}(K)$ and $\mathcal{D}_{p}^{n}(K)$}

Let $\widetilde{K}$ be the pre-image of $K$ in $\Sigma^{n}(K)$. Considering $\widetilde{K}$ as a knot in $\Sigma^{n}(K)$, Grigsby, Ruberman and Strle in \cite{MR2443966} defined the $\tau$-invariant $\tau (\widetilde{K}, s)$ for $\widetilde{K}$ and $s\in \mathrm{Spin}^{c}(\Sigma^{n}(K))$. This invariant satisfies the following property.
\begin{theo}[Grigsby, Ruberman and Strle]
\label{grs}
If $K$ is slice, then $\tau (\widetilde{K}, s_{0}+m)=0$ for any $m\in \mathrm{Ker}(\zeta)$.
\end{theo}

Note that the proof of Theorem~\ref{theo1} only depends on the algebraic information carried on $\mathrm{Ker}(\zeta)$ and Theorem~\ref{os}, while does not depend on the definition of $d$-invariant. By replacing Theorem~\ref{os} with Theorem~\ref{grs}, we can prove exactly the same fact for $\tau$-invariant as we did for $d$-invariant in Theorem~\ref{theo1}.

Grigsby, Ruberman and Strle in \cite{MR2443966} also defined invariants $\mathcal{D}_{p}(K)$ and $\mathcal{T}_{p}(K)$ associated with the double branched cyclic cover of the knot $K$. We can extend their definition naturally to the case of any $n$-branched cyclic cover with $n$ a prime power.

Suppose $\phi :E\rightarrow \mathbb{Q}$ is a function on a finite abelian group $E$ and $H<E$ is a subgroup. Following \cite{MR2443966} we let $S_{H}(\phi)=\sum_{h\in H}(\phi (h))$. Given a prime number $p$, let $\mathcal{G}_{p}$ be the set of all order $p$ subgroups of $H_{1}(\Sigma^{n}(K); \mathbb{Z})$. Then we can define
\begin{equation*}
\mathcal{T}_{p}^{n}(K)=\begin{cases}
\mathrm{min} \left\{ \abs{ \sum_{H\in \mathcal{G}_{p}}n_{H}S_{H}(\tau (\widetilde{K}, \cdot))} \left|
\begin{array}{l}
n_{H}\in \mathbb{Z}_{\geqslant 0} \text{ \& at least}\\
\text{one is non-zero} 
\end{array} \right.\right\}\\
\hspace{50mm} \text{ ; if $p$ divides $\vert H_{1}(\Sigma^{n}(K); \mathbb{Z})\vert$}\\
\hspace{10mm} 0 \hspace{38mm}  \text{ ; otherwise}
\end{cases}
\end{equation*}
and 
\begin{equation*}
\mathcal{D}_{p}^{n}(K)=\begin{cases}
\mathrm{min} \left\{ \abs{ \sum_{H\in \mathcal{G}_{p}}n_{H}S_{H}(d(\Sigma^{n}(K), \cdot))} \left|
\begin{array}{l}
n_{H}\in \mathbb{Z}_{\geqslant 0} \text{ \& at least}\\
\text{one is non-zero} 
\end{array} \right.\right\}\\
\hspace{50mm} \text{ ; if $p$ divides $\vert H_{1}(\Sigma^{n}(K); \mathbb{Z})\vert$}\\
\hspace{10mm} 0 \hspace{38mm}  \text{ ; otherwise}
\end{cases}.
\end{equation*}
Here we regard $\tau (\widetilde{K}, \cdot)$ and $d(\Sigma^{n}(K), \cdot)$ as functions from $H_{1}(\Sigma^{n}(K); \mathbb{Z})$ to $\mathbb{Q}$.

Given a function $\phi :E\rightarrow \mathbb{Q}$, we define $\bar{\phi}: E\rightarrow \mathbb{Q}$ by sending $e\in E$ to $\phi (e)-\phi (0)$. Let $\overline{\mathcal{T}}_{p}^{n}(K)$ and $\overline{\mathcal{D}}_{p}^{n}(K)$ be the invariants defined by taking $\bar{\tau}$ and $\bar{d}$. We prove the following theorem:

\begin{theo}
\label{2}
Let $p$ be a positive prime number or 1.
Suppose the Alexander polynomials of $K_{1}$ and $K_{2}$ are relatively prime in $\mathbb{Q}[t, t^{-1}]$. Suppose further that at least one of $K_{1}$ and $K_{2}$ has non-singular Seifert form. 
\begin{enumerate}
\item If $n_{1}K_{1}\sharp n_{2}K_{2}$ is a slice knot for some non-zero $n_{1}$ and $n_{2}$, then for all but finitely many primes $q$, the following holds:  $\mathcal{\overline{T}}_{p}^{n}(K_{i})=\mathcal{\overline{D}}_{p}^{n}(K_{i})=0$ for $i=1,2$, where $n$ is a power of $q$.

\item If $n_{1}K_{1}\sharp n_{2}K_{2}$ is a ribbon knot for some non-zero $n_{1}$ and $n_{2}$, the conclusions above hold for any prime power $n$.
\end{enumerate}
\end{theo}

\begin{proof}
The proof is a combination of the proof of Theorem~\ref{theo1}, the proof of Theorem~1.2 in \cite{MR2443966}, Proposition~3.4 in \cite{MR2443966},  Theorem~\ref{os} and Theorem~\ref{grs}.
\end{proof}

\section{Application}
It is known in \cite[Corollary 23]{MR0253348} that the twist knot $T_{k}$ is 
\begin{itemize}
\item of infinite order in the algebraic concordance group $C_{alg}$ if $k<0$;
\item algebraically slice if $k\geq 0$ and $4k+1$ is a square;
\item of finite order in $C_{alg}$ otherwise.
\end{itemize}

\begin{figure}[h]
	\centering
		\includegraphics[width=0.6\textwidth]{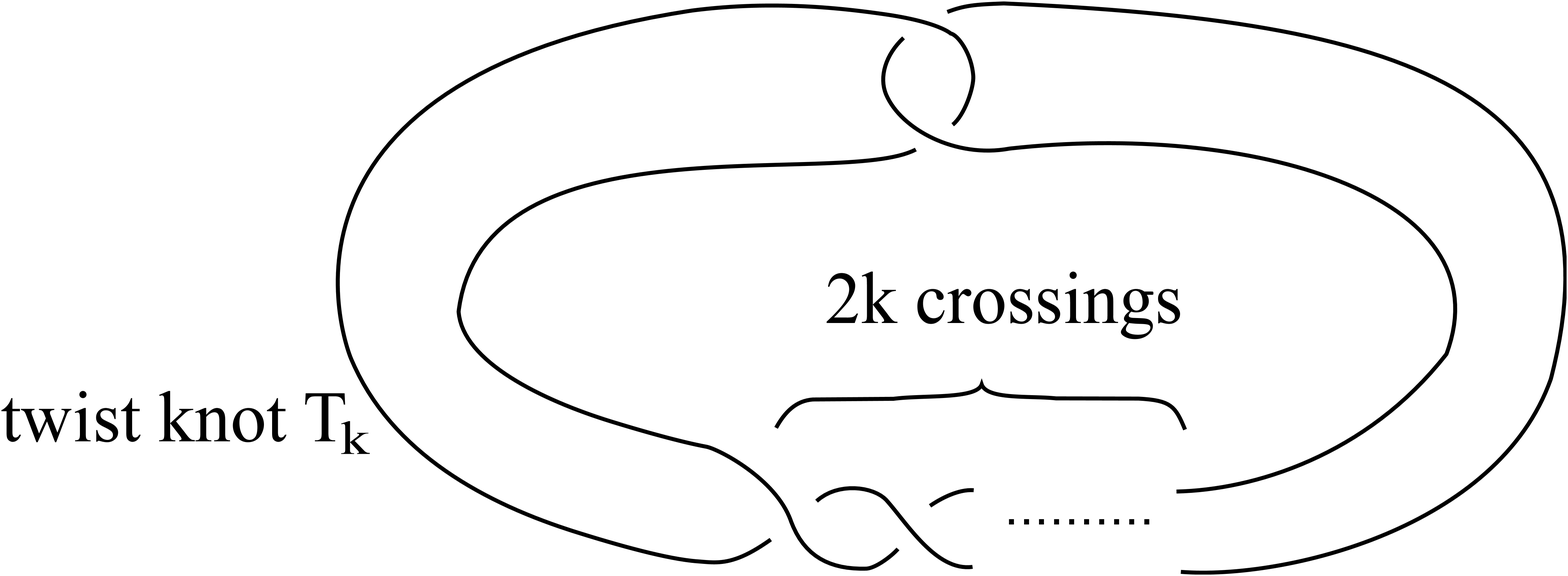}
		\caption{}
	\label{fig:fig4}
\end{figure}

%In this section, we show the following property.
%\begin{prop}
%Let $T_{k}$ be the $k$-twist knot. Excluding the unknot, $T_{1}$ (which is the figure-8 knot) and $T_{2}$ (which is Stevedore's knot), no non-trivial linear combinations of twist knots are ribbon knots.

%\end{prop}

\begin{proof}[Proof of Proposition~\ref{p}]
The Alexander polynomial of the $T_{k}$ is $$\Delta_{T_{k}}(t)=-kt^{2}+(2k+1)t-k.$$
It is easy to check that for any two non-trivial twist knots, their Alexander polynomials are relatively prime in $\mathbb{Q}[t^{-1}, t]$.

Note that each non-trival twist knot has non-singular Seifert form. Excluding the unknot, 1-twist knot and 2-twist knot, suppose that there are two sets of positive integers $\{ k_{i}\}_{i=1}^{l}$ and $\{ n_{i}\}_{i=1}^{l}$ for which $K=\sharp_{i=1}^{l}(n_{i}T_{k_{i}})$ is a ribbon knot. Then
\begin{itemize}
\item by Lemma~\ref{kim}, each $n_{i}T_{k_{i}}$ is algebraically slice;
\item by our Theorem~\ref{2}, each $T_{k_{i}}$ has vanishing $\mathcal{\bar{D}}_{p}^{q}$ and $\mathcal{\bar{T}}_{p}^{q}$ for any prime number $p$ and prime number $q$.
\end{itemize}

We consider the case when $q=2$, namely the double branched covers of the twist knots.

Recall that $T_{k}$ has infinite order in the algebraic concordance group $C_{alg}$ if $k<0$. So each $k_{i}$ for $1\leq i \leq l$ is a positive integer.

Let $L_{k}$ be the 3-manifold $\Sigma^{2}(T_{k})=L(4k+1,2)$. Assume that $k\geq 0$ and let $p$ be a prime number dividing $4k+1$. Then
\begin{eqnarray*}
\overline{\mathcal{D}}_{p}^{2}(T_{k}) =\left| \sum_{j=0}^{p-1} \bar{d}(L_{k}, s_{0}+j) \right|
= \left| \sum_{j=0}^{p-1} ({d}(L_{k}, s_{0}+j)-d(L_{k}, s_{0})) \right|.
\end{eqnarray*}

Ozsv{\'a}th and Szab{\'o} \cite{MR1957829} provided a formula of the $d$-invariants for lens spaces, by which we have
\begin{eqnarray*}
d(L_{k}, s_{0}+j)=\dfrac{1}{4}-\dfrac{j^{2}}{8k+2}+\begin{cases}
\dfrac{1}{4} & \text{ if $j$ is odd;}\\
\dfrac{-1}{4} & \text{ if $j$ is even,}
\end{cases}
\end{eqnarray*}
for $0\leq j \leq 2k$.

By calculation we have
$d(L_{k}, s_{0})=0$. So 
\begin{eqnarray*}
\overline{\mathcal{D}}_{p}^{2}(T_{k}) 
= \left| \sum_{j=0}^{p-1} {d}(L_{k}, s_{0}+j) \right|= {\mathcal{D}}_{p}^{2}(T_{k}).
\end{eqnarray*}

In \cite[Proposition~5.1]{MR2443966}, the authors discussed ${\mathcal{D}}_{p}^{2}(T_{k})$ for $k>0$ and showed that
$${\mathcal{D}}_{p}^{2}(T_{k})>0$$ except for the cases $k=0,1,2$.
Therefore those $T_{k_{i}}$ which make $\bar{\mathcal{D}}_{p}^{2}$ vanishes are restricted to $T_{0}, T_{1}$ and $T_{2}$. This completes the proof.
\end{proof}

\bibliographystyle{siam}
\bibliography{bao}

\end{document}